\documentclass[10pt, reqno]{amsart}

\def\BibTeX{{\rm B\kern-.05em{\sc i\kern-.025em b}\kern-.08em
    T\kern-.1667em\lower.7ex\hbox{E}\kern-.125emX}}

\newtheorem{thm}{Theorem}[section]

\newtheorem{lem}[thm]{Lemma}

\theoremstyle{definition}

\theoremstyle{remark}
\newtheorem{rem}{Remark}[section]
\newtheorem{exmp}{Example}[section]

\numberwithin{equation}{section}

    \newcommand{\floor}[1]{\lfloor#1\rfloor}

    \newcommand{\EE}{\mathbb{E}}

    \renewcommand{\Pr}{\operatorname{P}}

    \newcommand{\dto}{\xrightarrow{d}}
    
    \newcommand{\vto}{\xrightarrow{v}}
    \newcommand{\fidi}{\xrightarrow{\text{fidi}}}

    \newcommand{\toi}{\to\infty}

    \newcommand{\eind}{\stackrel{d}{=}}
    \newcommand{\rmd}{\mathrm{d}}

\newcommand{\be}{\begin{equation}}
    \newcommand{\ee}{\end{equation}}

\begin{document}

\title[A functional limit theorem for moving averages] 
{A functional limit theorem for moving averages with weakly dependent heavy-tailed innovations}

%
\author{Danijel Krizmani\'{c}}

\address{Danijel Krizmani\'{c}\\ Department of Mathematics\\
        University of Rijeka\\
        Radmile Matej\v{c}i\'{c} 2, 51000 Rijeka\\
        Croatia}
\email{dkrizmanic@math.uniri.hr}





\begin{abstract}
Recently a functional limit theorem for sums of moving averages with random coefficients and i.i.d.~heavy tailed innovations has been obtained under the assumption that all partial sums of the series of coefficients are a.s.~bounded between zero and the sum of the series. The convergence takes place in the space $D[0,1]$ of c\`{a}dl\`{a}g functions with the Skorohod $M_{2}$ topology. In this article we extend this result to the case when the innovations are weakly dependent in the sense of strong mixing and local dependence condition $D'$.
\end{abstract}

\maketitle

\section{Introduction}
\label{intro}

Let $(Z_{i})_{i \in \mathbb{Z}}$ be a strictly stationary sequence of regularly varying random variables with index of regular variation $\alpha \in (0,2)$.
This means that
\begin{equation}\label{e:regvar}
 \Pr(|Z_{i}| > x) = x^{-\alpha} L(x), \qquad x>0,
\end{equation}
where $L$ is a slowly varying function at $\infty$. Let $(a_{n})$ be a sequence of positive real numbers such that
\be\label{eq:niz}
n \Pr (|Z_{1}|>a_{n}) \to 1,
\ee
as $n \to \infty$. Then regular
variation of $Z_{i}$ can be expressed in terms of
vague convergence of measures on $\EE = \overline{\mathbb{R}} \setminus \{0\}$:
\begin{equation}
  \label{eq:onedimregvar}
  n \Pr( a_n^{-1} Z_i \in \cdot \, ) \vto \mu( \, \cdot \,) \qquad \textrm{as} \ n \to \infty,
\end{equation}
where $\mu$ is a measure on $\EE$  given by
\begin{equation}
\label{eq:mu}
  \mu(\rmd x) = \bigl( p \, 1_{(0, \infty)}(x) + r \, 1_{(-\infty, 0)}(x) \bigr) \, \alpha |x|^{-\alpha-1} \, \rmd x,
\end{equation}
with
\be\label{eq:pq}
p =   \lim_{x \to \infty} \frac{\Pr(Z_i > x)}{\Pr(|Z_i| > x)} \qquad \textrm{and} \qquad
  r =   \lim_{x \to \infty} \frac{\Pr(Z_i \leq -x)}{\Pr(|Z_i| > x)}.
\ee
We study
the moving average process with random coefficients, defined by
\begin{equation}\label{e:MArandom}
X_{i} = \sum_{j=0}^{\infty}C_{j}Z_{i-j}, \qquad i \in \mathbb{Z},
\end{equation}
where
$(C_{i})_{i \geq 0 }$ is a sequence of random variables independent of $(Z_{i})$ such that the above series is a.s.~convergent. One sufficient condition for that, which is commonly used in the literature is
\begin{equation}\label{e:momcond}
\sum_{j=0}^{\infty} \mathrm{E} |C_{j}|^{\delta} < \infty \qquad \textrm{for some}  \ \delta < \alpha,\,0 < \delta \leq 1.
\end{equation}
The moment condition (\ref{e:momcond}), stationarity of the sequence $(Z_{i})$ and $\mathrm{E}|Z_{1}|^{\beta} < \infty$ for every $\beta \in (0,\alpha)$ (which follows from the regular variation property and Karamata's theorem) imply the a.s.~convergence of the series in (\ref{e:MArandom}), since
$$ \mathrm{E}|X_{i}|^{\delta} \leq \sum_{j=0}^{\infty} \mathrm{E}|C_{j}|^{\delta} \mathrm{E}|Z_{i-j}|^{\delta} = \mathrm{E}|Z_{1}|^{\delta} \sum_{j=0}^{\infty}\mathrm{E}|C_{j}|^{\delta} < \infty.$$
Another condition that assures the a.s.~convergence of the series in the definition of linear processes with
$$\begin{array}{rl}
 \nonumber \mathrm{E}(Z_{1})=0, & \quad \textrm{if} \ \alpha \in (1,2),\\[0.2em]
 \nonumber Z_{1} \ \textrm{is symmetric}, & \quad \textrm{if} \ \alpha =1,
\end{array}$$
 and a.s.~bounded coefficients
 can be deduced from the results in Astrauskas~\cite{At83} for linear processes with deterministic coefficients:
$$ \sum_{j =0}^{\infty}c_{j}^{\alpha} L(c_{j}^{-1}) < \infty,$$
where $(c_{j})$ is a sequence of positive real numbers such that $|C_{j}| \leq c_{j}$ a.s.~for all $j$ (c.f.~ Balan et al.~\cite{BJL16}).

In the case when the $Z_{i}$'s are independent, under some usual regularity conditions and the assumption that all partial sums of the series $C= \sum_{i=0}^{\infty}C_{i}$ are a.s.~bounded between zero and the sum of the series, i.e.
\be\label{eq:InfiniteMAcond}
0 \le \sum_{i=0}^{s}C_{i} \Bigg/ \sum_{i=0}^{\infty}C_{i} \le 1 \ \ \textrm{a.s.} \qquad \textrm{for every} \ s=0, 1, 2 \ldots,
\ee
a functional limit theorem for the corresponding partial sum stochastic process
\be\label{eq:defVn1}
V_{n}(t) = \frac{1}{a_{n}}  \sum_{i=1}^{\floor {nt}}X_{i}, \qquad t \in [0,1],
\ee
in the space $D[0,1]$ with the Skorohod $M_{2}$ topology was recently derived in Krizmani\'{c}~\cite{Kr19}. More precisely,
\be\label{eq:fconvVn}
 V_{n}(\,\cdot\,) \dto \widetilde{C} V(\,\cdot\,) \qquad \textrm{as} \ n \to \infty,
\ee
in $D[0,1]$ endowed with the $M_{2}$ topology,
where $V$ is an $\alpha$--stable L\'{e}vy process with characteristic triple $(0, \mu, b)$, with $\mu$ as in $(\ref{eq:mu})$,
$$ b = \left\{ \begin{array}{cc}
                                   0, & \quad \alpha = 1,\\[0.4em]
                                   (p-r)\frac{\alpha}{1-\alpha}, & \quad \alpha \in (0,1) \cup (1,2),
                                 \end{array}\right.$$
    $\widetilde{C}$ is a random variable, independent of $V$, such that $\widetilde{C} \eind C$, and $D[0,1]$ is the space of real--valued right continuous functions on $[0,1]$ with left limits.
 When the sequence of coefficients $(C_{j})$ is deterministic, relation (\ref{eq:fconvVn}) reduces to
 $$ V_{n}(\,\cdot\,) \dto C V(\,\cdot\,) \qquad \textrm{as} \ n \to \infty$$
 (see Proposition 3.2 in Krizmani\'{c}~\cite{Kr19}). This functional convergence, as shown by Avram and Taqqu~\cite{AvTa92}, can not be strengthened to the Skorohod $J_{1}$ convergence on $D[0,1]$, but if all coefficients are nonnegative, then it holds in the $M_{1}$ topology.

 More precisely, let $X_{i} = \sum_{j \in \mathbb{Z}}c_{j}Z_{i-j}$ be a linear process with independent, regularly varying innovations $Z_{i}$ with index of regular variation $\alpha \in (0,2)$, and deterministic coefficients that are summable:
 $ \sum_{j \in \mathbb{Z}}|c_{j}| < \infty$.
 Assume also $\mathrm{E}(Z_{1})=0$ if $\alpha \in (1,2)$, and $Z_{1}$ is symmetric if $\alpha =1$.
 Then it is known that
 \begin{equation}\label{e:fidic}
   V_{n}(\,\cdot\,) \fidi \bigg( \sum_{j \in \mathbb{Z}}c_{j} \bigg) V(\,\cdot\,) \qquad \textrm{as} \ n \to \infty,
 \end{equation}
 where $V$ is an $\alpha$--stable L\'{e}vy process and ``$\fidi$" denotes convergence of finite-dimensional
distributions (see Astrauskas~\cite{At83}, Theorem 1i, and Balan et al.~\cite{BJL16}, Theorem 2.1). Avram and Taqqu~\cite{AvTa92} in their Theorem 1 showed that in the case of finite-order moving average with at least two non-zero coefficients the convergence in (\ref{e:fidic}) does not hold in the Skorohod $J_{1}$ topology (when only one coefficient is non-zero, the $J_{1}$ convergence holds by Skorohod~\cite{Sk57}). They also showed that the Skorohod $M_{1}$ convergence holds if one imposes some additional assumptions: if all coefficients are nonnegative, then in the case of finite-order moving averages the convergence in (\ref{e:fidic}) can be strengthened to the $M_{1}$ convergence on $D[0,1]$. The same holds for infinite-order moving averages when $\alpha \leq 1$, while when $\alpha >1$ under additional technical assumptions the $M_{1}$ convergence holds:
$$ \sum_{j \in \mathbb{Z}}|c_{j}|^{\nu} < \infty \qquad \textrm{for some} \ 0 < \nu < \alpha, \ \nu \leq 1,$$
and
$$ \lim_{n \to \infty} (\ln n)^{1+\alpha+\eta} \bigg( \sum_{|j| >n}|c_{j}|^{\nu} \bigg) \bigg( \sum_{|j|>n}|c_{j}| \bigg)^{\alpha-\eta-\nu}=0$$
for some $ 0 < \eta \leq \alpha-1$. This last condition holds if $\nu<1$ and $(c_{j})_{j \geq 0}$ and $(c_{j})_{j <0}$ are monotone sequences (see Avram and Taqqu~\cite{AvTa92}, Theorem 2', Proposition 1 and Theorem 2), but the results by Louhichi and Rio~\cite{LR11} show that it can be dropped. With less restrictive assumptions on the coefficients Basrak and Krizmani\'{c}~\cite{BaKr} obtained convergence in a weaker topology: if $c_{j}=0$ for $j<0$, $c_{0}, c_{1}, \ldots \in \mathbb{R}$ and for every $s=0,1,2,\ldots$
$$ 0 \leq \sum_{j=0}^{s}c_{j} \bigg/ \sum_{j=0}^{\infty}c_{j} \leq 1$$
(i.e.~all partial sums of the series of coefficients are bounded between zero and the sum of the series), then (\ref{e:fidic}) holds in the $M_{2}$ topology, the weakest of the four Skorohod topologies. Recently, Balan et al.~\cite{BJL16} obtained functional convergence in the $S$ topology under every of the the following three sets of conditions: (i) $\alpha \in (1,2)$ and $\sum_{j\in \mathbb{Z}}|c_{j}|< \infty$; (ii) $\alpha \leq 1$, $\sum_{j\in \mathbb{Z}}|c_{j}|^{\alpha}< \infty$ and the function $L$ from (\ref{e:regvar}) satisfies $L(\lambda x) / L(x) \leq M$ for $\lambda >1$ and $x \geq x_{0}$ (for some constants $M$, $x_{0}$); (iii) $\alpha < 1$, $\sum_{j\in \mathbb{Z}}|c_{j}|^{\alpha}< \infty$ and there exists a constant $0 < \gamma < \alpha$ such that
$$ \frac{\max_{j+1 \leq k \leq j+n}|c_{k}|^{\frac{(1-\alpha)(\alpha-\gamma)}{1-\alpha+\gamma}}}{\sum_{k=j+1}^{j+n}|c_{k}|^{\alpha}} \leq K_{+} < \infty, \quad j \geq 0,$$
$$ \frac{\max_{j-n \leq k \leq j-1}|c_{k}|^{\frac{(1-\alpha)(\alpha-\gamma)}{1-\alpha+\gamma}}}{\sum_{k=j-n}^{j-1}|c_{k}|^{\alpha}} \leq K_{-} < \infty, \quad j \leq 0$$
(with the convention that $0/0 \equiv 1$). The $S$ topology, introduced in Jakubowski~\cite{Ja97}, is a sequential and non-metric topology, weaker than the $M_{1}$ topology.

 In this paper we aim to extend the functional convergence in (\ref{eq:fconvVn}) to the case when the innovations $Z_{i}$ are weakly dependent, i.e.~$(Z_{i})$ is a strongly mixing sequence which satisfies the local dependence condition $D'$ as is given in Davis~\cite{Da83}:
 \begin{equation}\label{e:D'cond}
 \lim_{k \to \infty} \limsup_{n \to \infty}~n \sum_{i=1}^{\lfloor n/k \rfloor} \Pr \bigg( \frac{|Z_{0}|}{a_{n}} > x, \frac{|Z_{i}|}{a_{n}} >x \bigg) = 0 \qquad \textrm{for all} \ x >0.
 \end{equation}
 For instance, a process which is an instantaneous function of a stationary Gaussian process with covariance function $r_{n}$ behaving like $r_{n} \log n \to 0$ as $n \to \infty$ satisfies Condition (1.11), see Davis~\cite{Da83}. Other examples of time series that satisfy Condition (1.11), related to stochastic volatility models and ARMAX processes, can be found in Davis and Mikosch~\cite{DaMi09} and Ferreira and Canto e Castro~\cite{FeCa08}. This condition, together with the strong mixing property, assures that, as in the i.i.d.~case, the extremes of the sequence $(Z_{i})$ are isolated. This corresponds to the situation when the extremal index $\theta$ of the sequence $(Z_{i})$, which can be interpreted as the reciprocal mean cluster size of large exceedances (c.f.~Hsing et al.~\cite{HHL88}), is equal to $1$. When $\theta < 1$ clustering of extreme values occurs, and in general condition (\ref{e:D'cond}) and the convergence in (\ref{eq:fconvVn}) fail to hold, see Example~\ref{ex:M2fails} below for an illustration. Recall here that a strictly stationary sequence of random variables $(\xi_{n})$ has extremal index $\theta$ if for every $\tau >0$ there exists a sequence of real numbers $(u_{n})$ such that
\begin{equation*}\label{eq:eindex}
 \lim_{n \to \infty} n \Pr ( \xi_{1} > u_{n}) \to \tau \qquad \textrm{and} \qquad \lim_{n \to \infty} \Pr \bigg( \max_{1 \leq i \leq n} \xi_{i} \leq u_{n} \bigg) \to e^{-\theta \tau}.
\end{equation*}
It holds that $\theta \in [0,1]$. Recall also that a sequence $(\xi_{n})$ is strongly mixing if $\alpha (n) \to 0$ as $n \to \infty$, where
$$\alpha (n) = \sup \{|\Pr (A \cap B) - \Pr(A) \Pr(B)| : A \in \mathcal{F}_{1}^{j}, B \in \mathcal{F}_{j+n}^{\infty}, j=1,2, \ldots \}$$
and $\mathcal{F}_{k}^{l} = \sigma( \{ \xi_{i} : k \leq i \leq l \} )$ for $1 \leq k \leq l \leq \infty$.
For some related results regarding functional convergence for linear type processes with short memory under the assumptions that the innovations are dependent and the coefficients deterministic, we refer to Tyran-Kami\'{n}ska~\cite{Ty10b}. For some related results on limit theory for moving averages with random coefficients we refer to Hult and Samorodnitsky~\cite{HuSa08} and Kulik~\cite{Ku06}.


We also impose the following standard regularity conditions on $Z_{1}$:
  \begin{eqnarray}\label{e:oceknula}
    \mathrm{E} Z_{1}=0, & & \textrm{if} \ \ \alpha \in (1,2),  \\
    Z_{1} \ \textrm{is symmetric}, & & \textrm{if} \ \ \alpha=1.\label{e:sim}
  \end{eqnarray}
Beside condition (\ref{e:momcond}) we will require some other moment conditions, which will be specified in Section~\ref{S:InfiniteMA}.
Further, in the case $\alpha \in [1,2)$ we will need to assume the following condition:
 \begin{equation}\label{e:vsvcond}
 \lim_{u \downarrow 0} \limsup_{n \to \infty}~\Pr \bigg[ \max_{1 \leq k \leq n} \bigg| \sum_{i=1}^{k}  \bigg( \frac{Z_{i}}{a_{n}} 1_{\{|Z_{i}|/a_{n} \leq u\}} - \mathrm{E} \bigg( \frac{Z_{i}}{a_{n}} 1_{\{ |Z_{i}|/a_{n} \leq u \}} \bigg) \bigg) \bigg| > \epsilon \bigg]=0
 \end{equation}
for all $\epsilon >0$. This condition holds if the sequence $(Z_{i})$ is $\rho$-mixing at a certain rate (see Lemma 4.8 in Tyran-Kami\'{n}ska~\cite{Ty10a}). In case $\alpha \in (0,1)$ it is a simple consequence of regular variation and Karamata's theorem. Similar conditions are used often in the related literature on the limit theory for partial sums, see~\cite{AvTa92, BKS, DuRe78, Ty10a}.

 The Skorohod $M_{2}$ topology on $D[0, 1]$ is defined using completed graphs and their parametric representations (see Section 12.11 in Whitt~\cite{Whitt02} for details). We will use the following characterization of the $M_{2}$ topology with the Hausdorff metric on the spaces of graphs: for $x_{1},x_{2} \in D[0,1]$, the $M_{2}$ distance between $x_{1}$ and $x_{2}$ is given by
$$ d_{M_{2}}(x_{1}, x_{2}) = \bigg(\sup_{a \in \Gamma_{x_{1}}} \inf_{b \in \Gamma_{x_{2}}} d(a,b) \bigg) \vee \bigg(\sup_{a \in \Gamma_{x_{2}}} \inf_{b \in \Gamma_{x_{1}}} d(a,b) \bigg),$$
where $\Gamma_{x}$ is the completed graph of $x \in D[0,1]$ defined by
 \[
  \Gamma_{x}
  = \{ (t,z) \in [0,1] \times \mathbb{R} : z= \lambda x(t-) + (1-\lambda)x(t) \ \text{for some}\ \lambda \in [0,1] \},
\]
where $x(t-)$ is the left limit of $x$ at $t$, $d$ is the metric on $\mathbb{R}^{2}$ defined by $d((x_{1},y_{1}),(x_{2},y_{2}))=|x_{1}-x_{2}| \vee |y_{1}-y_{2}|$ for $(x_{i},y_{i}) \in \mathbb{R}^{2},\,i=1,2$, and $a \vee b = \max\{a,b\}$.
The metric $d_{M_{2}}$ induces the $M_{2}$ topology, which is weaker than the more frequently used $M_{1}$ and $J_{1}$ topologies.

The paper is organized as follows. In Section~\ref{S:FiniteMA} we derive functional convergence of partial sum stochastic processes for finite order moving averages, and then in Section~\ref{S:InfiniteMA} we
extend this result to infinite order moving average processes.

\section{Finite order MA processes}
\label{S:FiniteMA}

Fix $q \in \mathbb{N}$ and let $C_{0}, C_{1}, \ldots , C_{q}$ be random variables satisfying
\be\label{eq:FiniteMAcond}
0 \le \sum_{i=0}^{s}C_{i} \Bigg/ \sum_{i=0}^{q}C_{i} \le 1 \ \ \textrm{a.s.} \qquad \textrm{for every} \ s=0, 1, \ldots, q.
\ee
Condition (\ref{eq:FiniteMAcond}) implies that $ C = \sum_{i=0}^{q}C_{i}$,
$ \sum_{i=0}^{s}C_{i}$ and $ \sum_{i=s}^{q}C_{i}$ are a.s.~of the same sign for every $ s=0,1,\ldots,q$. Note that condition (\ref{eq:FiniteMAcond}) is satisfied if the $C_{j}$'s are all nonnegative or all nonpositive.

Let $(X_{t})$ be a moving average process defined by
$$ X_{t} = \sum_{i=0}^{q}C_{i}Z_{t-i}, \qquad t \in \mathbb{Z},$$
and let the corresponding partial sum process be
\be\label{eq:defVn}
V_{n}(t) = \frac{1}{a_{n}}  \sum_{i=1}^{\floor {nt}}X_{i}, \qquad t \in [0,1],
\ee
where the normalizing sequence $(a_n)$ satisfies~\eqref{eq:niz}.

\begin{thm}\label{t:FinMA}
Let $(Z_{i})_{i \in \mathbb{Z}}$ be a strictly stationary and strongly mixing sequence of regularly varying random variables with index $\alpha \in (0,2)$, such that conditions $(\ref{e:D'cond})$, $(\ref{e:oceknula})$ and $(\ref{e:sim})$ hold. If $\alpha \in [1,2)$, also suppose that condition $(\ref{e:vsvcond})$ holds.
Assume $C_{0}, C_{1}, \ldots , C_{q}$ are random variables, independent of $(Z_{i})$, that satisfy
$(\ref{eq:FiniteMAcond})$. Then
$$ V_{n}(\,\cdot\,) \dto \widetilde{C} V(\,\cdot\,) \qquad \textrm{as} \ n \to \infty,$$
in $D[0,1]$ endowed with the $M_{2}$ topology, where $V$ is an $\alpha$--stable L\'{e}vy process with characteristic triple $(0, \mu, b)$, with $\mu$ as in $(\ref{eq:mu})$,
$$ b = \left\{ \begin{array}{cc}
                                   0, & \quad \alpha = 1,\\[0.4em]
                                   (p-r)\frac{\alpha}{1-\alpha}, & \quad \alpha \in (0,1) \cup (1,2),
                                 \end{array}\right.$$
and $\widetilde{C}$ is a random variable, independent of $V$, such that $\widetilde{C} \eind C$.
\end{thm}

In the proof of the above theorem we will use the following lemma (which can be proven as in Basrak and Krizmani\'{c}~\cite{BaKr}), functional convergence for regularly varying time series with isolated extremes and appropriate modifications of Theorem 2.1 in Krizmani\'{c}~\cite{Kr19}.

\begin{lem}\label{l:first}
With the notation $C_{i}=0$ for $i<0$, it holds that:
\begin{itemize}
  \item[(i)] For $k < q$
  \begin{eqnarray*}
  \sum_{i=1}^{k}\frac{C\,Z_{i}}{a_{n}} - \sum_{i=1}^{k}\frac{X_{i}}{a_{n}} & = & \sum_{u=0}^{k-1}\frac{Z_{k-u}}{a_{n}} \sum_{s=u+1}^{q}C_{s} - \sum_{u=k-q}^{q-1}\frac{Z_{-u}}{a_{n}} \sum_{s=u+1}^{q}C_{s}\\[0.6em]
  & & - \sum_{u=0}^{q-k-1}\frac{Z_{-u}}{a_{n}} \sum_{s=u+1}^{u+k}C_{s}.
  \end{eqnarray*}
  \item[(ii)] For $k \ge q$
  \begin{eqnarray*}
  \sum_{i=1}^{k}\frac{C\,Z_{i}}{a_{n}} - \sum_{i=1}^{k}\frac{X_{i}}{a_{n}} & = & \sum_{u=0}^{q-1}\frac{Z_{k-u}}{a_{n}} \sum_{s=u+1}^{q}C_{s} - \sum_{u=0}^{q-1}\frac{Z_{-u}}{a_{n}} \sum_{s=u+1}^{q}C_{s}\\[0.7em]
  & =: & H_{n}^{k} - G_{n}.
  \end{eqnarray*}
  \item[(iii)] For $q \le k \le n-q$
  \begin{eqnarray*}
  \sum_{i=1}^{k}\frac{C\,Z_{i}}{a_{n}} - \sum_{i=1}^{k+q}\frac{X_{i}}{a_{n}} & = & - \sum_{u=0}^{q-1}\frac{Z_{-u}}{a_{n}} \sum_{s=u+1}^{q}C_{s} - \sum_{u=1}^{q}\frac{Z_{k+u}}{a_{n}} \sum_{s=0}^{q-u}C_{s}\\[0.7em]
  & =: & -G_{n} - T_{n}^{k}.
  \end{eqnarray*}
\end{itemize}
\end{lem}

\begin{proof} ({\it Theorem~\ref{t:FinMA}})
Condition (\ref{e:D'cond}) and the strong mixing property imply that the extremes of the sequence $(Z_{i})$ are isolated, i.e.~$\theta=1$ (see Leadbetter and Rootz\'{e}n~\cite{LeRo88}, page 439, and Leadbetter et al.~\cite{LLR83}, Theorem 3.4.1) and the tail process $(Y_{i})$ of the sequence $(Z_{i})$ defined by $\Pr(|Y_{0}|>y)=y^{-\alpha}$ for $y \geq 1$ and
 \begin{equation*}\label{e:tailprocess}
  \bigl( (x^{-1}Z_i)_{i \in \mathbb{Z}} \, \big| \, |Z_0| > x \bigr)
  \fidi (Y_i)_{i \in \mathbb{Z}} \quad \textrm{as} \ x \to \infty,
\end{equation*}
is the same as for an i.i.d.~sequence, that is, $Y_{i}=0$ for $i \neq 0$, and $Y_{0}$ is as described above, see Basrak et al.~\cite{BKS}.
As a special case of their main theorem on functional $M_{1}$ convergence of partial sum processes of stationary, regularly varying sequences for which all extremes within each cluster of large values have the same sign (i.e.~the corresponding tail process almost surely has no two values of the opposite sign), Basrak et al.~\cite{BKS} obtained $M_{1}$ convergence for processes with isolated extremes. More precisely, they showed that for strictly stationary and strongly mixing sequences of regularly varying random variables that satisfy the dependence condition (\ref{e:D'cond}) and the vanishing small values condition (\ref{e:vsvcond}), the properly centered partial sum process converges in distribution to an $\alpha$--stable L\'{e}vy process in $D[0,1]$ with the $M_{1}$ topology (see Basrak et al.~\cite{BKS}, Theorem 3.4 and Example 4.1). We apply this result directly to our case to conclude that, as $n \to \infty$,
\begin{equation*}
  \sum_{i=1}^{\floor {nt} }\frac{Z_{i}}{a_{n}} - \floor{nt} \mathrm{E} \Big(\frac{Z_{1}}{a_{n}} 1_{\{|Z_{1}| \leq a_{n}\}} \Big), \qquad t \in [0,1],
\end{equation*}
converges in distribution in $D[0,1]$ with the $M_{1}$ topology to an $\alpha$--stable L\'{e}vy process with characteristic triple $(0,\mu,0)$ with $\mu$ as in (\ref{eq:mu}).

By Karamata's theorem, as $n \to \infty$,
\begin{eqnarray*}
  n\,\mathrm{E} \Big( \frac{Z_{1}}{a_{n}} 1_{\{ |Z_{1}| \leq a_{n} \}} \Big) \to (p-r)\frac{\alpha}{1-\alpha}, && \textrm{if}  \ \ \alpha < 1,\\[0.5em]
  n\,\mathrm{E} \Big( \frac{Z_{1}}{a_{n}} 1_{\{ |Z_{1}| > a_{n} \}} \Big) \to (p-r)\frac{\alpha}{\alpha-1}, && \textrm{if} \ \ \alpha >1,
\end{eqnarray*}
with $p$ and $r$ as in (\ref{eq:pq}). Therefore conditions (\ref{e:oceknula}) and (\ref{e:sim}), Corollary 12.7.1 in Whitt~\cite{Whitt02} (which gives a sufficient condition for addition to be continuous in the $M_{1}$ topology) and the continuous mapping theorem yield that
$V_{n}^{Z}(\,\cdot\,) \dto V(\,\cdot\,)$ in $D[0,1]$ with the $M_{1}$ topology, where
$$ V_{n}^{Z}(t) := \sum_{i=1}^{\floor {nt}}\frac{Z_{i}}{a_{n}}, \qquad t \in [0,1],$$
and
 $V$ is an $\alpha$--stable L\'{e}vy process
 with characteristic triple
$(0,\mu,0)$ if $\alpha=1$ and $(0,\mu,(p-r)\alpha/(1-\alpha))$ if $\alpha \in (0,1) \cup (1,2)$.

Since the space $D[0,1]$ equipped with the $M_{1}$ topology is a Polish space (see Section 14 in Billingsley~\cite{Bi68} and Section 12.8 in Whitt~\cite{Whitt02}), by Corollary 5.18 in Kallenberg~\cite{Ka97}  we can find a random variable $\widetilde{C}$, independent of $V$, such that $\widetilde{C} \eind C$. This and the fact that $C$ is independent of $V_{n}^{Z}$, by an application of Theorem 3.29 in Kallenberg~\cite{Ka97}, imply
  \begin{equation}\label{e:zajedkonvK}
   (B(\,\cdot\,), V_{n}^{Z}(\,\cdot\,)) \dto (\widetilde{B}(\,\cdot\,), V(\,\cdot\,)), \qquad \textrm{as} \ n \to \infty,
  \end{equation}
  in $D([0,1], \mathbb{R}^{2})$ with the product $M_{1}$ topology, where $B(t)=C$ and $\widetilde{B}(t)=\widetilde{C}$ for $t \in [0,1]$.
Applying the continuous mapping theorem to relation (\ref{e:zajedkonvK}) we obtain
$ B(\,\cdot\,) V_{n}^{Z}(\,\cdot\,) \dto \widetilde{B}(\,\cdot\,) V(\,\cdot\,)$ as $n \to \infty$, i.e.
$ C V_{n}^{Z}(\,\cdot\,) \dto \widetilde{C} V(\,\cdot\,)$ in $D[0,1]$ with the $M_{1}$ topology. Since $M_{1}$ convergence implies $M_{2}$ convergence, we have
\be
CV_{n}^{Z}(\,\cdot\,) \dto \widetilde{C} V(\,\cdot\,), \qquad \textrm{as} \ n \to \infty,
\ee
in $(D[0,1], d_{M_{2}})$ as well. It remains to show that for every $\epsilon >0$
$$ \lim_{n \to \infty}\Pr[d_{M_{2}}(CV_{n}^{Z}, V_{n})> \epsilon]=0,$$
since then by an application of  Slutsky's theorem (see for instance Theorem 3.4 in Resnick~\cite{Resnick07})
it will follow that
$V_{n}(\,\cdot\,) \dto \widetilde{C} V(\,\cdot\,)$ in $(D[0,1], d_{M_{2}})$.

Fix $\epsilon >0$ and let $n \in \mathbb{N}$ be large enough, i.e.
 $n > \max\{2q, 2q/\epsilon \}$.
By the definition of the metric $d_{M_{2}}$ we have
\begin{eqnarray*}
  d_{M_{2}}(CV_{n}^{Z},V_{n}) &=& \bigg(\sup_{a \in \Gamma_{CV_{n}^{Z}}} \inf_{b \in \Gamma_{V_{n}}} d(a,b) \bigg) \vee \bigg(\sup_{a \in \Gamma_{V_{n}}} \inf_{b \in \Gamma_{CV_{n}^{Z}}} d(a,b) \bigg) \\[0.4em]
   &= :& Y_{n} \vee T_{n},
\end{eqnarray*}
and therefore
\be\label{eq:AB}
\Pr [d_{M_{2}}(CV_{n}^{Z}, V_{n})> \epsilon ] \leq \Pr(Y_{n}>\epsilon) + \Pr(T_{n}>\epsilon)\,.
\ee

By the same arguments as in the proof of Theorem 2.1 in Krizmani\'{c}~\cite{Kr19} (see also Basrak and Krizmani\'{c}~\cite{BaKr}) for the first term on the right hand side of (\ref{eq:AB}) we have
\begin{eqnarray}\label{eq:Yn}
  \nonumber\{Y_{n} > \epsilon\} & \subseteq & \{\exists\,a \in \Gamma_{CV_{n}^{Z}} \ \textrm{such that} \ d(a,b) > \epsilon \ \textrm{for every} \ b \in \Gamma_{V_{n}} \} \\[0.6em]
  \nonumber & \subseteq & \{\exists\,k \in \{1,\ldots,q-1\} \ \textrm{such that} \ | CV_{n}^{Z}(k/n) - V_{n}(k/n)| > \epsilon \}\\[0.6em]
  \nonumber & & \cup \ \{\exists\,k \in \{q,\ldots,n-q\} \ \textrm{such that} \ | CV_{n}^{Z}(k/n) - V_{n}(k/n)| > \epsilon\\[0.6em]
  \nonumber & & \hspace*{1.5em} \textrm{and} \  | CV_{n}^{Z}(k/n) - V_{n}((k+q)/n)| > \epsilon \}\\[0.6em]
  \nonumber & & \cup \ \{\exists\,k \in \{n-q+1,\ldots,n\} \ \textrm{such that} \ | CV_{n}^{Z}(k/n) - V_{n}(k/n)| > \epsilon \}\\[0.6em]
  & =: & A^{Y}_{n} \cup B^{Y}_{n} \cup C^{Y}_{n}.
\end{eqnarray}

By Lemma~\ref{l:first} (i) we obtain
\begin{eqnarray}\label{eq:Bnlemmafirst}
  \nonumber \Pr (A^{Y}_{n}) & \leq &  \sum_{k=1}^{q-1} \Pr \Big( \Big| \sum_{i=1}^{k}\frac{C Z_{i}}{a_{n}} - \sum_{i=1}^{k}\frac{X_{i}}{a_{n}} \Big| >\epsilon \Big) \\[0.6em]
    \nonumber  & \leq & \sum_{k=1}^{q-1} \bigg[ \Pr \Big( \sum_{u=0}^{k-1}\frac{|Z_{k-u}|}{a_{n}}
         \sum_{s=u+1}^{q}|C_{s}| > \frac{\epsilon}{3} \Big) + \Pr \Big( \sum_{u=k-q}^{q-1}\frac{|Z_{-u}|}{a_{n}} \sum_{s=u+1}^{q}|C_{s}| > \frac{\epsilon}{3} \Big)\\[0.6em]
  \nonumber  & & \hspace*{2em} + \Pr \Big( \sum_{u=0}^{q-k-1}\frac{|Z_{-u}|}{a_{n}} \sum_{s=u+1}^{u+k}|C_{s}| > \frac{\epsilon}{3} \Big) \bigg]\\[0.6em]
      & \leq &  3(q-1)(2q-1) \Pr \Big( \frac{|Z_{0}|}{a_{n}}\,C_{*} > \frac{\epsilon}{3(2q-1)} \Big),
  \end{eqnarray}
   where $C_{*} = \sum_{s=0}^{q}|C_{s}|$. Take now $M>0$ arbitrary and note
   \begin{eqnarray*}
     \Pr \Big( \frac{|Z_{0}|}{a_{n}}\,C_{*} > \frac{\epsilon}{3(2q-1)} \Big) &&  \\[0.7em]
      &\hspace*{-16em} =&  \hspace*{-8em} \Pr \Big( \frac{|Z_{0}|}{a_{n}}\,C_{*} > \frac{\epsilon}{3(2q-1)},\,C_{*} > M \Big) + \Pr \Big( \frac{|Z_{0}|}{a_{n}}\,C_{*} > \frac{\epsilon}{3(2q-1)},\,C_{*} \leq M \Big)\\[0.7em]
       &\hspace*{-16em} \leq &  \hspace*{-8em} \Pr \Big( C_{*} > M \Big) + \Pr \Big( \frac{|Z_{0}|}{a_{n}} > \frac{\epsilon}{3(2q-1)M} \Big).
   \end{eqnarray*}
Since by the
 regular variation property
it holds that
\begin{equation*}
\lim_{n \to \infty} \Pr \Big( \frac{|Z_{0}|}{a_{n}} > \frac{\epsilon}{3(2q-1)M} \Big) =0,
\end{equation*}
 from (\ref{eq:Bnlemmafirst}) we get
$$ \limsup_{n \to \infty} \Pr (A^{Y}_{n}) \leq \Pr \Big(C_{*} > M \Big).$$
Letting $M \to \infty$ we conclude
 \be\label{eq:setBn1}
 \lim_{n \to \infty} \Pr (A^{Y}_{n}) = 0.
 \ee
Similarly
  \be\label{eq:setBn4}
   \lim_{n \to \infty} \Pr (C^{Y}_{n})=0.
  \ee
Next, using Lemma~\ref{l:first} (ii) and (iii), for an arbitrary $M>0$
we obtain
\begin{eqnarray}\label{e:setBn1new}
   \nonumber \Pr (B^{Y}_{n}\cap \{ C_{*} \leq M \}) & = &  \Pr \Big( \exists\,k \in \{q,\ldots,n-q\} \ \textrm{such that} \ |H_{n}^{k}-G_{n}| > \epsilon \\[0.6em]
   \nonumber  & & \hspace*{2em} \textrm{and} \ |-G_{n}-T_{n}^{k}| > \epsilon,\,C_{*} \leq M \Big)\\[0.6em]
   & \hspace*{-19em} \leq & \hspace*{-9.5em} \Pr \Big( |G_{n}|> \frac{\epsilon}{2},\,C_{*} \leq M \Big) + \sum_{k=q}^{n-q} \Pr \Big( |H_{n}^{k}| >
       \frac{\epsilon}{2},\,|T_{n}^{k}| > \frac{\epsilon}{2},\,C_{*} \leq M \Big).
\end{eqnarray}
Note that
\begin{eqnarray*}
  \Pr \Big( |G_{n}|> \frac{\epsilon}{2},\,C_{*} \leq M \Big) & \leq & \Pr \Big( C_{*} \sum_{u=0}^{q-1}\frac{|Z_{-u}|}{a_{n}}> \frac{\epsilon}{2},\,C_{*} \leq M \Big)  \\[0.5em]
   & \leq &  \Pr \Big( \sum_{u=0}^{q-1}\frac{|Z_{-u}|}{a_{n}}> \frac{\epsilon}{2M} \Big)\\[0.5em]
   & \leq &  q \Pr \Big( \frac{|Z_{0}|}{a_{n}} > \frac{\epsilon}{2qM} \Big),
\end{eqnarray*}
and an application of the regular variation property yields
\begin{equation}\label{e:setBn2new}
 \lim_{n \to \infty} \Pr \Big( |G_{n}|> \frac{\epsilon}{2},\,C_{*} \leq M \Big) = 0.
\end{equation}
Further, since
$$
 H_n^{k}= \sum_{u=0}^{q-1}\frac{Z_{k-u}}{a_{n}} \sum_{s=u+1}^{q}C_{s}
\ \mbox{ and } \ T_n^{k} = \sum_{u=1}^{q}\frac{Z_{k+u}}{a_{n}}
 \sum_{s=0}^{q-u}C_{s},
 $$
  for a fixed $k \in \{q,\ldots,n-q\}$, on the event
$ \{ |H_{n}^{k}|
> \epsilon/2  \ \textrm{and} \
|T_{n}^{k}|
> \epsilon/2,\,C_{*} \leq M \}$
 there exist $i \in \{k-(q-1),\ldots,k\}$ and $j \in \{k+1,\ldots,k+q\}$ such that
 $$\frac{|Z_{i}|}{a_{n}} >  \frac{\epsilon}{2qM} \quad \textrm{and} \quad \frac{|Z_{j}|}{a_{n}}
> \frac{\epsilon}{2qM}.$$
Therefore, using the stationarity of the sequence $(Z_{i})$ we obtain
\begin{eqnarray*}
  \Pr \Big( |H_{n}^{k}| >
       \frac{\epsilon}{2} \ \textrm{and} \ |T_{n}^{k}| > \frac{\epsilon}{2},\,C_{*} \leq M \Big) & &\\[0.6em]
   & \hspace*{-26em} \leq & \hspace*{-13em} \sum_{\scriptsize \begin{array}{c}
                          i=k-(q-1),\ldots,k \\
                          j=k+1,\ldots,k+q
                        \end{array}} \Pr \Big( \frac{|Z_{i}|}{a_{n}} >
       \frac{\epsilon}{2qM}, \frac{|Z_{j}|}{a_{n}} > \frac{\epsilon}{2qM}\Big)\\[0.6em]
   & \hspace*{-26em} \leq & \hspace*{-13em} q \sum_{j=1}^{2q-1} \Pr \Big( \frac{|Z_{0}|}{a_{n}} >
       \frac{\epsilon}{2qM}, \frac{|Z_{j}|}{a_{n}} > \frac{\epsilon}{2qM}\Big),
\end{eqnarray*}
and hence for all positive integers $s \leq n/(2q-1)$ it holds
\begin{eqnarray*}
  \sum_{k=q}^{n-q} \Pr \Big( |H_{n}^{k}| >
       \frac{\epsilon}{2} \ \textrm{and} \ |T_{n}^{k}| > \frac{\epsilon}{2},\,C_{*} \leq M \Big) & &  \\[0.6em]
   & \hspace*{-22em} \leq & \hspace*{-11em} n q \sum_{j=1}^{\lfloor n/s \rfloor} \Pr \Big( \frac{|Z_{0}|}{a_{n}} >
       \frac{\epsilon}{2qM}, \frac{|Z_{j}|}{a_{n}} > \frac{\epsilon}{2qM}\Big).
\end{eqnarray*}
From this, taking into account condition (\ref{e:D'cond}), we conclude that
\begin{equation*}\label{e:setBn3new}
 \lim_{n \to \infty} \sum_{k=q}^{n-q} \Pr \Big( |H_{n}^{k}| >
       \frac{\epsilon}{2} \ \textrm{and} \ |T_{n}^{k}| > \frac{\epsilon}{2},\,C_{*} \leq M \Big) = 0.
\end{equation*}
Together with relations (\ref{e:setBn1new}) and (\ref{e:setBn2new}) this implies
 $$ \lim_{n \to \infty}  \Pr (B^{Y}_{n}\cap \{ C_{*} \leq M \})=0.$$
 Thus
 $$ \limsup_{n \to \infty} \Pr (B^{Y}_{n}) \leq \limsup_{n \to \infty} \Pr (B^{Y}_{n}\cap \{ C_{*} > M \} \leq \Pr (C_{*} > M),$$
and letting again $M \to \infty$ we conclude
 \be\label{eq:setBn3}
 \lim_{n \to \infty} \Pr (B^{Y}_{n}) = 0.
 \ee
  From relations (\ref{eq:Yn}), (\ref{eq:setBn1}), (\ref{eq:setBn4}) and (\ref{eq:setBn3}) we obtain
  \be\label{eq:Ynend}
  \lim_{n \to \infty} \Pr(Y_{n} > \epsilon ) =0.
  \ee

 In order to estimate the second term on the right hand side of (\ref{eq:AB}) define
  for each $k \geq q$ the numbers $V^{Z,\min}_k = \min\{CV^Z_n((k-q)/n), CV^Z_n(k/n) \}$
and $V^{Z,\max}_k = \max\{CV^Z_n((k-q)/n), CV^Z_n(k/n) \}$. Following the arguments in the proof of Theorem 2.1 in Krizmani\'{c}~\cite{Kr19} we obtain
\begin{eqnarray}\label{eq:Zn}
  \nonumber\{T_{n} > \epsilon\} & \subseteq & \{\exists\,a \in \Gamma_{V_{n}} \ \textrm{such that} \ d(a,b) > \epsilon \ \textrm{for every} \ b \in \Gamma_{CV_{n}^{Z}} \} \\[0.6em]
  \nonumber & \subseteq & \{\exists\,k \in \{1,\ldots,2q-1\} \
\ \textrm{such that} \ | V_{n}(k/n) - CV_{n}^{Z}(k/n)| > \epsilon \}\\[0.6em]
  \nonumber & & \cup \ \Big\{\exists\,k \in \{2q,\ldots,n\} \
   \textrm{such that} \
  \widetilde{d}(V_{n}(k/n), [V_{k}^{Z, \min}, V_{k}^{Z, \max}]) > \epsilon
    \Big\}\\[0.6em]
  & =: & A^{T}_{n} \cup B^{T}_{n},
\end{eqnarray}
where $\widetilde{d}$ is the Euclidean metric on $\mathbb{R}$.
Using Lemma~\ref{l:first} (i) and (ii), one could similarly as before for the set $A^{Y}_{n}$ obtain
\be\label{eq:Cnfirst}
\lim_{n \to \infty} \Pr( A^{T}_{n})=0.
\ee
 Note that $\Pr (B_n^T)$ is bounded above by
\begin{eqnarray*}
\lefteqn{ \Pr \left(
\exists\,k \in \{2q,\ldots,n\} \ \textrm{such that} \
\sum_{i=1}^k \frac{X_i}{a_n} > V^{Z,\max}_k + \epsilon \right)}\\
&+&
\Pr \left(
\exists\,k \in \{2q,\ldots,n\} \ \textrm{such that} \
\sum_{i=1}^k \frac{X_i}{a_n} < V^{Z,\min}_k - \epsilon \right)\,.
\end{eqnarray*}
We consider only the first of these two probabilities,
since the other one can be handled in a similar manner.
Using Lemma~\ref{l:first} the first probability can be bounded  by
\begin{eqnarray*}
\lefteqn{\Pr \left(
\exists\,k \in \{2q,\ldots,n\} \ \textrm{such that} \
G_n - H_n^{k} > \epsilon \ \mbox{ and }\
G_n + T_n^{k-q} > \epsilon   \right)} \\[0.3em]
& \leq &
\Pr \left(G_n  >  \frac{\epsilon}{2}\right)\\[0.3em]
&&  +
\Pr \left(
\exists\,k \in \{2q,\ldots,n\} \ \textrm{such that} \
H_n^{k} < -  \frac{\epsilon}{2} \ \mbox{ and }\
T_n^{k-q} >  \frac{\epsilon}{2}   \right)\,.
\end{eqnarray*}
From the calculations yielding (\ref{eq:setBn3}) we conclude that $\Pr (G_n  > \epsilon/2 )\to 0$
 as $n \to \infty$. The second term is bounded by
\begin{equation}\label{eq:BnTpom}
\Pr(C_{*}> M) + \sum_{k=2q}^{n} \Pr \left(
H_{n}^{k} < - \frac{\epsilon}{2}
\ \mbox{ and }
T_{n}^{k-q} > \frac{\epsilon}{2},\,C_{*} \leq M
 \right)
\end{equation}
for an arbitrary $M>0$. Note that
$$
 H_n^{k}= \sum_{u=0}^{q-1}\frac{Z_{k-u}}{a_{n}} \sum_{s=u+1}^{q}C_{s}
\ \mbox{ and } \ T_n^{k-q} = \sum_{u=0}^{q-1}\frac{Z_{k-u}}{a_{n}}
 \sum_{s=0}^{u}C_{s}.
 $$
Hence for a fixed $k \in \{2q,\ldots,n\}$, on the event
$ \{ H_{n}^{k}
< - \epsilon/2  \ \textrm{and} \
T_{n}^{k-q}
> \epsilon/2,\,C_{*} \leq M \}$
 there exist $i,j \in \{0,\ldots,q-1\}$ such that
 $$\frac{Z_{k-i}}{a_{n}} \sum_{s=i+1}^{q}C_{s}
< - \frac{\epsilon}{2q} \quad \textrm{and} \quad \frac{Z_{k-j}}{a_{n}}
 \sum_{s=0}^{j}C_{s}
> \frac{\epsilon}{2q}.$$
Condition \eqref{eq:FiniteMAcond} implies the sums
$\sum_{s=0}^{j}C_{s}$
and $ \sum_{s=i+1}^{q}C_{s}$ are a.s.~of the same sign, and since
their absolute values are bounded by $C_{*}$, we obtain
$|Z_{k-i}|M/a_{n} > \epsilon/(2q)$ and
$|Z_{k-j}|M/a_{n} > \epsilon/(2q)$.
The case $i=j$ is not possible since then we would have $Z_{k-i}<0$ and $Z_{k-i}>0$. From this, using the stationarity of the sequence $(Z_{i})$, we conclude that the expression in (\ref{eq:BnTpom}) is bounded by
\begin{eqnarray*}\label{e:RVdet2}
\nonumber & & \hspace*{-2em} {\Pr(C_{*}> M) +  n \Pr\left(
\exists\, i,j \in \{0,\ldots,q-1\},\,i \neq j \ \textrm{s.t.} \
M \frac{|Z_{-i }|}{a_{n}} >  \frac{\epsilon}{2q}
 \mbox{ and } M \frac{|Z_{-j }|}{a_{n}} >  \frac{\epsilon}{2q} \right)} \\
& \leq & \Pr(C_{*}> M) + 2n (q-1) \sum_{j=1}^{q-1}
  \Pr\left(
 \frac{|Z_{0}|}{a_{n}} >   \frac{\epsilon}{2 q M}, \frac{|Z_{j}|}{a_{n}} >   \frac{\epsilon}{2 q M}
 \right),
\end{eqnarray*}
which, similarly as before when considering the set $B_{n}^{Y}$, tends to 0 if we first let $n\toi$ and then $M \to \infty$.
Together with relations (\ref{eq:Zn}) and (\ref{eq:Cnfirst}) this implies
\be\label{eq:Tnend}
\lim_{n \to \infty} \Pr(T_{n}>\epsilon)=0.
\ee
Now from (\ref{eq:AB}), (\ref{eq:Ynend}) and (\ref{eq:Tnend}) we obtain
\be
\lim_{n \to \infty} \Pr [d_{M_{2}}(CV_{n}^{Z}, V_{n})> \epsilon ]=0,
\ee
and finally we conclude that $V_{n}(\,\cdot\,) \dto \widetilde{C} V(\,\cdot\,)$, as $n \to \infty$, in $(D[0,1], d_{M_{2}})$.
This concludes the proof.
\end{proof}

\begin{exmp}\label{ex:M2fails}
Condition (\ref{e:D'cond}) prohibits clustering of extreme values in the sequence $(Z_{i})$, which means that it has extremal index $\theta=1$. Here we give an example when clustering of extreme values occurs, and all conditions in Theorem~\ref{t:FinMA} hold except condition (\ref{e:D'cond}), but the convergence of the partial sum stochastic process $V_{n}$, as defined in (\ref{eq:defVn}), in $D[0,1]$ with the $M_{2}$ topology fails to hold.

Let $(\xi_{i})_{i}$ be a sequence of i.i.d.~regularly varying random variables with index of regular variation $\alpha \in (0,2)$. Define
$$ Z_{i}= \xi_{i} + \xi_{i-2}, \qquad i \in \mathbb{Z},$$
and assume conditions $(\ref{e:oceknula})$ and $(\ref{e:sim})$ hold.
The sequence $(Z_{i})$ is strictly stationary and consists of regularly varying random variables (see Proposition 7.4 in Resnick~\cite{Resnick07} and Theorem 1.28 in Lindskog~\cite{Lindskog04}). Consider the finite order moving average process
$$ X_{t}= Z_{t}-Z_{t-1}+Z_{t-2}, \qquad t \in \mathbb{Z}.$$
Hence $q=2$, $C_{0}=1$, $C_{1}=-1$, $C_{2}=1$, and condition $(\ref{eq:FiniteMAcond})$ clearly holds.
 Let $(a_{n})$ be a sequence of positive real numbers for which (\ref{eq:niz}) holds. By Lemma 1.2 in Cline~\cite{Cl83}
$$ \lim_{x \to \infty} \frac{\Pr(|Z_{0}|>x)}{\Pr(|\xi_{0}|>x)}=3,$$
which yields
\begin{equation}\label{eq:nizxi}
\lim_{n \to \infty} n \Pr(|\xi_{0}|>a_{n}) = 1/3.
\end{equation}
This together with the regular variation property of $\xi_{i}$ and the following inequality
\begin{eqnarray*}
 n \Pr \Big( \frac{|Z_{0}|}{a_{n}} >x, \frac{|Z_{2}|}{a_{n}} >x \Big) & \geq & n \Pr \Big( \frac{|\xi_{0}|}{a_{n}} > 2x, \frac{|\xi_{2}|}{a_{n}} \leq x, \frac{|\xi_{-2}|}{a_{n}} \leq x \Big)\\[0.4em]
  & = & n \Pr \Big( \frac{|\xi_{0}|}{a_{n}} > 2x \Big) \Big[ \Pr \Big( \frac{|\xi_{0}|}{a_{n}} \leq x \Big) \Big]^{2},
 \end{eqnarray*}
 implies
$$ \liminf_{n \to \infty} n \Pr \Big( \frac{|Z_{0}|}{a_{n}} >x, \frac{|Z_{2}|}{a_{n}} >x \Big) \geq \frac{1}{3} (2x)^{-\alpha} >0,$$
for $x>0$, and therefore we conclude that condition (\ref{e:D'cond}) does not hold. The sequence $(Z_{i})$ is $2$--dependent, and hence strongly mixing, with extremal index $\theta= 1/2$ (see Embrechts et al.~\cite{EmKlMi97}, page 415).

Next we show that $V_{n}$ does not converge in distribution under the $M_{2}$ topology on $D[0,1]$. For this, according to Skorohod~\cite{Sk56} (cf.~Proposition 2 in Avram and Taqqu~\cite{AvTa92}), it suffices to show that
 \begin{equation}\label{e:osc1}
 \lim_{\delta \to 0} \limsup_{n \to \infty} \Pr [ \triangle_{M_{2}}(\delta, V_{n}) > \epsilon ] > 0
 \end{equation}
 for some $\epsilon >0$, where
 $$ \triangle_{M_{2}}(\delta, x) = \sup_{{\footnotesize \begin{array}{c}
                                0 \leq t \leq 1 \\
                                t_{\delta} \leq t_{1} \leq t_{\delta} +\delta/2\\
                                t_{\delta}^{*} -\delta/2 \leq t_{2} \leq t_{\delta}^{*}
                              \end{array}}
} M(x(t_{1}), x(t), x(t_{2}))$$
($x \in D[0,1], \delta >0)$, $t_{\delta}=\max\{0, t-\delta \}$, $t_{\delta}^{*}=\min\{1, t+ \delta \}$, and
$$ M(x_{1},x_{2},x_{3}) = \left\{ \begin{array}{ll}
                                   0, & \ \ \textrm{if} \ x_{2} \in [x_{1}, x_{3}], \\
                                   \min\{ |x_{2}-x_{1}|, |x_{3}-x_{2}| \}, & \ \ \textrm{otherwise}.
                                 \end{array}\right.$$
Note that $M(x_{1},x_{2},x_{3})$ is the distance from $x_{2}$ to $[x_{1}, x_{3}]$, and $\triangle_{M_{2}}(\delta, x)$ is the $M_{2}$ oscillation of $x$.
To show (\ref{e:osc1}) we use, with appropriate modifications, the procedure of Avram and Taqqu~\cite{AvTa92} in the proof of their Theorem 1.

Let $i'=i'(n)$ be the index at which $\max_{1 \leq i \leq n-2}|\xi_{i}|$ is obtained. Fix $\epsilon >0$ and introduce the events
 $$A_{n,\epsilon} = \{ |\xi_{i'}| > \epsilon a_{n} \} = \Big\{ \max_{1 \leq i \leq n-2} |\xi_{i}| > \epsilon a_{n}\Big\}$$
 and
 $$ B_{n,\epsilon} = \{ |\xi_{i'}| >\epsilon a_{n} \ \textrm{and} \ \exists\,l
 \neq 0, -i'-3 \leq l \leq 1, \ \textrm{such that} \ |\xi_{i'+l}| > \epsilon a_{n} / 8 \}.$$
Using the facts that $(\xi_{i})$ is an i.i.d.~sequence and $n
 \Pr (|\xi_{0}|> \lambda a_{n}) \to \lambda^{-\alpha}/2$ as $n \to \infty$, for $\lambda >0$
 (which follows from the regular variation property of $\xi_{0}$ and (\ref{eq:nizxi})) we get
\begin{equation}\label{e:limAn}
   \Pr (A_{n,\epsilon}) = 1- \bigg[ 1 - \frac{n \Pr(|\xi_{0}|>\epsilon a_{n})}{n} \bigg]^{n-2} \to 1 - e^{-\epsilon^{-\alpha}/2},
\end{equation}
as $n \to \infty$,
and
 \begin{eqnarray}\label{e:limBn}
\nonumber \limsup_{n \to \infty} \Pr (B_{n,\epsilon})  & \leq & \limsup_{n \to \infty} \Pr \bigg( \bigcup_{i=1}^{n-2} \bigcup_{\footnotesize \begin{array}{c}
                                l=-(n-2)-3 \\
                                l \neq 0
                              \end{array}}^{1} \{ |\xi_{i}|>\epsilon a_{n}, |\xi_{i+l}|>\epsilon a_{n}/8\} \bigg)\\[0.4em]
\nonumber & \leq & \limsup_{n \to \infty}\,(n-2)(n+2) \Pr (|\xi_{0}|>\epsilon a_{n}) \Pr(|\xi_{0}| > \epsilon a_{n}/8)\\[0.4em]
 & = & \frac{\epsilon^{-2\alpha}}{4 \cdot 8^{-\alpha}}.
 \end{eqnarray}
On the event $A_{n,\epsilon} \setminus B_{n,\epsilon}$
one has $|\xi_{i'}| > \epsilon a_{n}$ and $|\xi_{i'+l}| \leq \epsilon a_{n}/8$ for every $l \in \{-i'-3, \ldots 1\} \setminus \{0\}$, and hence since
$$ V_{n}(t) = \sum_{i=1}^{\lfloor nt \rfloor}\frac{X_{i}}{a_{n}} = \frac{1}{a_{n}} \Big[ \xi_{\lfloor nt \rfloor} + 2 \xi_{\lfloor nt \rfloor -2} + \xi_{\lfloor nt \rfloor -3} + 2( \xi_{\lfloor nt \rfloor -4} + \ldots + \xi_{1}) +  \xi_{0}  + 2 \xi_{-1} + \xi_{-3} \Big],$$
we have
\begin{eqnarray}\label{e:inc1}
  \nonumber \Big| V_{n} \Big( \frac{i'}{n} \Big) - V_{n} \Big( \frac{i'-1}{n} \Big) \Big| & = & \frac{|\xi_{i'}-\xi_{i'-1}+2\xi_{i'-2}-\xi_{i'-3}+\xi_{i'-4}|}{a_{n}}\\[0.3em]
   & >&  \epsilon - \frac{5\epsilon}{8}= \frac{3\epsilon}{8}
\end{eqnarray}
and
\begin{eqnarray}\label{e:inc2}
  \Big| V_{n} \Big( \frac{i'+1}{n} \Big) - V_{n} \Big( \frac{i'}{n} \Big) \Big| &=& \frac{|\xi_{i'+1}-\xi_{i'}+2\xi_{i'-1}-\xi_{i'-2}+\xi_{i'-3}|}{a_{n}}\\[0.3em]
   & >& \epsilon - \frac{5\epsilon}{8}=\frac{3\epsilon}{8}.
\end{eqnarray}
Further, on the event $A_{n,\epsilon} \setminus B_{n,\epsilon}$ it also holds that
\begin{equation*}\label{e:inc3}
 V_{n} \Big( \frac{i'}{n} \Big) \notin \Big[ V_{n} \Big( \frac{i'-1}{n} \Big), V_{n} \Big( \frac{i'+1}{n} \Big) \Big],
\end{equation*}
since
$$ \max \bigg\{  V_{n} \Big( \frac{i'-1}{n} \Big),  V_{n} \Big( \frac{i'+1}{n} \Big)  \bigg\} <  V_{n} \Big( \frac{i'}{n} \Big) \qquad \textrm{if} \ \xi_{i'}>0,$$
and
$$\min \bigg\{  V_{n} \Big( \frac{i'-1}{n} \Big),  V_{n} \Big( \frac{i'+1}{n} \Big)  \bigg\} >  V_{n} \Big( \frac{i'}{n} \Big) \qquad \textrm{if} \ \xi_{i'}<0.$$
Therefore
\begin{eqnarray*}
  M \Big( V_{n} \Big( \frac{i'-1}{n} \Big), V_{n} \Big( \frac{i'}{n} \Big), V_{n} \Big( \frac{i'+1}{n} \Big) \Big) & &  \\[0.8em]
   & \hspace*{-20em} =& \ \hspace*{-10em} \min \bigg\{  \Big| V_{n} \Big( \frac{i'}{n} \Big) - V_{n} \Big( \frac{i'-1}{n} \Big) \Big|, \Big| V_{n} \Big( \frac{i'+1}{n} \Big) - V_{n} \Big( \frac{i'}{n} \Big) \Big| \bigg\}.
\end{eqnarray*}
Taking into account (\ref{e:inc1}) and (\ref{e:inc2}) we obtain
\begin{eqnarray*}
  \triangle_{M_{2}}(1/n, V_{n}) & = & \sup_{{\footnotesize \begin{array}{c}
                                0 \leq t \leq 1 \\
                                t_{1/n} \leq t_{1} \leq t_{1/n} +1/2n\\
                                t_{1/n}^{*} -1/2n \leq t_{2} \leq t_{1/n}^{*}
                              \end{array}}
} M(V_{n}(t_{1}), V_{n}(t), V_{n}(t_{2})) \\[0.8em]
   & \geq & M \Big( V_{n} \Big( \frac{i'-1}{n} \Big), V_{n} \Big( \frac{i'}{n} \Big), V_{n} \Big( \frac{i'+1}{n} \Big) \Big) > \frac{3 \epsilon}{8}
\end{eqnarray*}
on the event $A_{n,\epsilon} \setminus B_{n,\epsilon}$. Therefore, since $\triangle_{M_{2}}(\delta, V_{n})$ is nondecreasing in $\delta$, it holds that
 \begin{eqnarray}\label{e:oscM1}
  \nonumber \liminf_{n \to \infty} \Pr (A_{n,\epsilon} \setminus B_{n,\epsilon}) & \leq & \liminf_{n \to \infty}
 \Pr [\triangle_{M_{2}} (1/n, V_{n}) >  3 \epsilon /8]\\[0.4em]
 & \leq &   \lim_{\delta \to 0} \limsup_{n \to \infty}  \Pr [\triangle_{M_{2}} (\delta, V_{n}) >  3 \epsilon/8].
 \end{eqnarray}
  Since $x^{2\alpha}(1-e^{-x^{-\alpha}/2})$ tends to infinity as $x \to \infty$, we can find $\epsilon >0$ such that $\epsilon^{2\alpha}(1-e^{-\epsilon^{-\alpha}/2}) > 8^{\alpha}/4$, i.e.
 $$ 1-e^{-\epsilon^{-\alpha}/2} > \frac{\epsilon^{-2\alpha}}{4 \cdot 8^{-\alpha}}.$$
 For this $\epsilon$, by relations (\ref{e:limAn}) and (\ref{e:limBn}), it holds that
 $$\lim_{n \to \infty} \Pr (A_{n,\epsilon}) > \limsup_{n \to \infty} \Pr (B_{n,\epsilon}),$$
 i.e.
 $$  \liminf_{n \to \infty} \Pr (A_{n,\epsilon} \setminus B_{n,\epsilon}) \geq \lim_{n \to \infty}\Pr (A_{n,\epsilon}) - \limsup_{n \to \infty} \Pr (B_{n,\epsilon}) >0.$$
 Therefore by (\ref{e:oscM1}) we obtain
$$ \lim_{\delta \to 0} \limsup_{n \to \infty}  \Pr [\triangle_{M_{2}} (\delta, V_{n}) >  3\epsilon/8] > 0$$
and relation (\ref{e:osc1}) holds, which means that $V_{n}$ does not converge in distribution in $D[0,1]$ endowed with the $M_{2}$ topology.

Using similar arguments, one can obtain the same conclusion for the moving average process
$$ X'_{t}= Z'_{t}+Z'_{t-1}, \qquad t \in \mathbb{Z},$$
where $Z'_{i}= \xi_{i} - \xi_{i-1}$, $i \in \mathbb{Z}$, but in this case the sequence $(Z'_{i})$ has extremal index $\theta= 1$. Therefore, even when clustering of extreme values do not occur, condition (\ref{e:D'cond}) and functional $M_{2}$ convergence of the corresponding partial sum stochastic process may fail to hold. Here the fact that $V_{n}$ does not converge in $D[0,1]$ with the $M_{2}$ topology can be seen also by the following reasoning.
Observe that $ X'_{i}= \xi_{i}-\xi_{i-2}$, and hence, as $n \to \infty$,
\begin{equation}\label{e:fidi}
V_{n}(\,\cdot\,) = \frac{1}{a_{n}}  \sum_{i=1}^{\floor {n\,\cdot}}X '_{i} = \frac{\xi_{ \floor{n \cdot}}+\xi_{ \floor{n \cdot}-1} - \xi_{0}-\xi_{-1}}{a_{n}} \fidi 0.
\end{equation}
Since, as is known, $\sup_{t \in [0,1]}\xi_{\floor{nt}}/a_{n}$ converges in distribution to a non-zero limit (cf.~Proposition 7.2 in Resnick~\cite{Resnick07}) and the functional $\sup_{t \in [0,1]}$ is continuous in the $M_{2}$ topology (see Skorohod~\cite{Sk57}), the ``fidi" convergence in (\ref{e:fidi}) can not be replaced by convergence in distribution under the $M_{2}$ topology (neither under the other Skorohod's topologies).
\end{exmp}

\section{Infinite order MA processes}
\label{S:InfiniteMA}

For infinite order moving averages the idea is to approximate them by a sequence of finite order moving averages, for which Theorem~\ref{t:FinMA} holds, and to show that the error of approximation is negligible in the limit. In the case $\alpha \in (0,1)$ one can follow directly the lines in the proof of Theorem 3.1 in Krizmani\'{c}~\cite{Kr19} to obtain the functional convergence of the corresponding partial sum stochastic processes. In the case $\alpha \in [1,2)$ the arguments from the proof of Theorem 3.1 in Krizmani\'{c}~\cite{Kr19} can not be applied to our setting, since due to the dependence in the sequence $(Z_{i})$, certain sequences constructed from $(C_{i})$ and $(Z_{i})$ are no longer martingales and martingale-difference sequences, which was crucial in obtaining functional convergence for infinite order moving averages with i.i.d.~innovations. The idea in this case is to use the arguments from Lemma 2 in Tyran-Kami\'{n}ska~\cite{Ty10b} to show that functional convergence of the partial sum stochastic processes still holds. More precisely, we have the following result.

\begin{thm}\label{t:InfMA}
Let $(X_{i})$ be a moving average process defined by
$$ X_{i} = \sum_{j=0}^{\infty}C_{j}Z_{i-j}, \qquad i \in \mathbb{Z},$$
 where $(Z_{i})_{i \in \mathbb{Z}}$ is a strictly stationary and strongly mixing sequence of regularly varying random variables with index $\alpha \in (0,2)$, such that conditions $(\ref{e:D'cond})$, $(\ref{e:oceknula})$ and $(\ref{e:sim})$ hold, and $(C_{i})_{i \geq 0}$ is a sequence of random variables, independent of $(Z_{i})$, satisfying conditions $(\ref{e:momcond})$ and $(\ref{eq:InfiniteMAcond})$.
If $\alpha \in (0,1)$ suppose further
\begin{equation}\label{e:mod1}
 \sum_{i=0}^{\infty}\mathrm{E}|C_{i}|^{\gamma} < \infty \qquad \textrm{for some} \ \gamma \in (\alpha, 1),
\end{equation}
while if $\alpha \in [1,2)$ suppose condition $(\ref{e:vsvcond})$ holds,
\begin{equation}\label{eq:infmaTK}
\limsup_{n \to \infty} \sup_{j \geq 0} \mathrm{E} \bigg[ \max_{1 \leq l \leq n} \bigg| \frac{1}{a_{n}} \sum_{i=1}^{l}Z_{i-j} 1_{\{ |Z_{i-j}| \leq a_{n} \}} \bigg|^{r}\bigg] < \infty \qquad \textrm{for some} \ r \geq 1,
\end{equation}
and
\begin{equation}\label{eq:infmaTK3}
\sum_{j=0}^{\infty} \mathrm{E}|C_{j}| < \infty.
\end{equation}
  Then
$$ V_{n}(\,\cdot\,) \dto \widetilde{C} V(\,\cdot\,), \qquad n \to \infty,$$
in $D[0,1]$ endowed with the $M_{2}$ topology, where $V$ is an $\alpha$--stable L\'{e}vy process with characteristic triple $(0, \mu, b)$, with $\mu$ as in $(\ref{eq:mu})$ and
$$ b = \left\{ \begin{array}{cc}
                                   0, & \quad \alpha = 1,\\[0.4em]
                                   (p-r)\frac{\alpha}{1-\alpha}, & \quad \alpha \in (0,1) \cup (1,2),
                                 \end{array}\right.$$
   and $\widetilde{C}$ is a random variable, independent of $V$, such that $\widetilde{C} \eind \sum_{i=0}^{\infty}C_{i}$.
\end{thm}
\begin{proof}
Take $q \in \mathbb{N}$ and define
$$ X_{i}^{q} = \sum_{j=0}^{q-1}C_{j}Z_{i-j} + C'_{q} Z_{i-q}, \qquad i \in \mathbb{Z},$$
where $C'_{q}= \sum_{i=q}^{\infty}C_{i}$,
and
$$ V_{n, q}(t) = \sum_{i=1}^{\floor{nt}} \frac{X_{i}^{q}}{a_{n}}, \qquad t \in [0,1].$$
The coefficients $C_{0}, \ldots, C_{q-1}, C'_{q}$ satisfy condition (\ref{eq:FiniteMAcond}), and hence an application of Theorem~\ref{t:FinMA} to a finite order moving average process $(X_{i}^{q})_{i}$ yields that, as $n \to \infty$,
$$V_{n, q}(\,\cdot\,) \dto \widetilde{C} V(\,\cdot\,)$$
in $(D[0,1], d_{M_{2}})$, where $V$ is an $\alpha$--stable L\'{e}vy process with characteristic triple as in Theorem~\ref{t:FinMA}
and $\widetilde{C}$ is a random variable, independent of $V$, such that $\widetilde{C} \eind \sum_{i=0}^{\infty}C_{i}$.

In the case $\alpha \in (0,1)$ by repeating the arguments from the proof of Theorem 3.1 in Krizmani\'{c}~\cite{Kr19} we obtain
\begin{equation}\label{e:infSlutsky}
 \lim_{q \to \infty} \limsup_{n \to \infty} \Pr[d_{M_{2}}(V_{n, q}, V_{n})> \epsilon]=0,
\end{equation}
for every $\epsilon >0$. Therefore, by a generalization of Slutsky's theorem (see for instance Theorem 3.5 in Resnick~\cite{Resnick07}) it follows that $V_{n}(\,\cdot\,) \dto \widetilde{C} V(\,\cdot\,)$, as $n \to \infty$, in $(D[0,1], d_{M_{2}})$.

Assume now $\alpha \in [1,2)$. We will use the arguments from the proof of Lemma 2 in Tyran-Kami\'{n}ska~\cite{Ty10b} adapted to linear processes with random coefficients instead of deterministic. Define $Z_{n,j} = a_{n}^{-1} Z_{j} 1_{\{ |Z_{j}| \leq a_{n} \}}$ for $j \in \mathbb{Z}$ and $n \in \mathbb{N}$,
$$ \widetilde{C}_{j} = \left\{ \begin{array}{cc}
                                   C_{j}, & \quad \textrm{if} \ j > q,\\[0.4em]
                                   C_{q}-C'_{q}, & \quad \textrm{if} \ j=q,
                                 \end{array}\right.$$
and note that
\begin{eqnarray*}
\nonumber  V_{n}(t) - V_{n,q}(t) & = & \sum_{i=1}^{\floor{nt}} \frac{1}{a_{n}} \bigg( \sum_{j=q}^{\infty}C_{j}Z_{i-j} - C'_{q}Z_{i-q} \bigg) \\[0.4em]
& = &  \sum_{i=1}^{\floor{nt}}  \sum_{j=q}^{\infty} \widetilde{C}_{j}Z_{n,i-j} + \sum_{i=1}^{\floor{nt}}  \sum_{j=q}^{\infty} \frac{\widetilde{C}_{j}Z_{i-j}}{a_{n}} 1_{\{ |Z_{i-j}| > a_{n} \}}.
\end{eqnarray*}
Since the Skorohod $M_{2}$ metric on $D[0,1]$ is bounded above by the uniform metric on $D[0,1]$, we have
\begin{eqnarray}\label{eq:I1I2}
\nonumber   \Pr[d_{M_{2}}(V_{n, q}, V_{n})> \epsilon] & \leq & \Pr \bigg( \sup_{0 \leq t \leq 1} |V_{n}(t) - V_{n,q}(t)| > \epsilon \bigg)\\[0.4em]
\nonumber & \hspace*{-18em} \leq & \hspace*{-9em} \Pr \bigg( \max_{1 \leq l \leq n} \bigg| \sum_{i=1}^{l} \sum_{j=q}^{\infty} \widetilde{C}_{j}Z_{n,i-j} \bigg| > \frac{\epsilon}{2} \bigg) + \Pr \bigg( \max_{1 \leq l \leq n} \bigg| \sum_{i=1}^{l} \sum_{j=q}^{\infty} \frac{\widetilde{C}_{j}Z_{i-j}}{a_{n}} 1_{\{ |Z_{i-j}|>a_{n} \}} \bigg| > \frac{\epsilon}{2} \bigg)\\[0.4em]
& \hspace*{-18em} =: & \hspace*{-9em} I_{1} + I_{2}.
\end{eqnarray}
By H\"{o}lder's inequality we have
$$ \bigg( \sum_{j=q}^{\infty} |\widetilde{C}_{j}| \cdot \bigg| \sum_{i=1}^{l}Z_{n,i-j} \bigg| \bigg)^{r} \leq \bigg( \sum_{j=q}^{\infty}|\widetilde{C}_{j}| \bigg)^{r-1} \sum_{j=q}^{\infty}|\widetilde{C}_{j}| \cdot \bigg| \sum_{i=1}^{l}Z_{n,i-j} \bigg|^{r},$$
with $r$ as in (\ref{eq:infmaTK}), and therefore using Markov's inequality we obtain
\begin{eqnarray*}
\nonumber I_{1} & \leq & \Pr \bigg( \sum_{j=q}^{\infty}|\widetilde{C}_{j}| >1 \bigg) + \Pr \bigg( \max_{1 \leq l \leq n} \bigg| \sum_{i=1}^{l} \sum_{j=q}^{\infty} \widetilde{C}_{j}Z_{n,i-j} \bigg| > \frac{\epsilon}{2},\,\sum_{j=q}^{\infty}|\widetilde{C}_{j}| \leq 1 \bigg)\\[0.4em]
& \leq & \mathrm{E} \bigg( \sum_{j=q}^{\infty}|\widetilde{C}_{j}| \bigg) + \Pr \bigg( \max_{1 \leq l \leq n} \bigg| \sum_{i=1}^{l} \sum_{j=q}^{\infty} \widetilde{C}_{j}Z_{n,i-j} \bigg|^{r} > \Big(\frac{\epsilon}{2}\Big)^{r},\,\sum_{j=q}^{\infty}|\widetilde{C}_{j}| \leq 1 \bigg)\\[0.4em]
& \leq & \mathrm{E} \bigg( \sum_{j=q}^{\infty}|\widetilde{C}_{j}| \bigg) + \Pr \bigg( \max_{1 \leq l \leq n} \bigg( \sum_{j=q}^{\infty}|\widetilde{C}_{j}| \bigg)^{r-1} \sum_{j=q}^{\infty}|\widetilde{C}_{j}| \cdot \bigg| \sum_{i=1}^{l}Z_{n,i-j} \bigg|^{r}  > \Big(\frac{\epsilon}{2}\Big)^{r},\,\sum_{j=q}^{\infty}|\widetilde{C}_{j}| \leq 1 \bigg)\\[0.4em]
& \leq & \mathrm{E} \bigg( \sum_{j=q}^{\infty}|\widetilde{C}_{j}| \bigg) + \Pr \bigg( \max_{1 \leq l \leq n}  \sum_{j=q}^{\infty}|\widetilde{C}_{j}| \cdot \bigg| \sum_{i=1}^{l}Z_{n,i-j} \bigg|^{r}  > \Big(\frac{\epsilon}{2}\Big)^{r},\,\sum_{j=q}^{\infty}|\widetilde{C}_{j}| \leq 1 \bigg).
\end{eqnarray*}
Now, using again Markov's inequality and the fact that the sequence $(C_{i})_{i \geq 0}$ is independent of $(Z_{i})$ we obtain
\begin{eqnarray*}
I_{1} & \leq &  \mathrm{E} \bigg( \sum_{j=q}^{\infty}|\widetilde{C}_{j}| \bigg) + \frac{2^{r}}{\epsilon^{r}}\mathrm{E} \bigg[ \sum_{j=q}^{\infty}|\widetilde{C}_{j}| \cdot \max_{1 \leq l \leq n} \bigg| \sum_{i=1}^{l}Z_{n,i-j} \bigg|^{r} \bigg]\\[0.4em]
 & \leq & \mathrm{E} \bigg( \sum_{j=q}^{\infty}|\widetilde{C}_{j}| \bigg) + \frac{2^{r}}{\epsilon^{r}} \sum_{j=q}^{\infty} \mathrm{E} |\widetilde{C}_{j}| \cdot \mathrm{E} \bigg( \max_{1 \leq l \leq n} \bigg| \sum_{i=1}^{l}Z_{n,i-j} \bigg|^{r} \bigg)\\[0.4em]
 & \leq & \mathrm{E} \bigg( \sum_{j=q}^{\infty}|\widetilde{C}_{j}| \bigg) + \frac{2^{r}}{\epsilon^{r}} \sum_{j=q}^{\infty} \mathrm{E} |\widetilde{C}_{j}| \cdot \sup_{k \geq q} \mathrm{E} \bigg( \max_{1 \leq l \leq n} \bigg| \sum_{i=1}^{l}Z_{n,i-k} \bigg|^{r} \bigg).
\end{eqnarray*}
Noting that $ \sum_{j=q}^{\infty}|\widetilde{C}_{j}| \leq 2 \sum_{j=q}^{\infty}|C_{j}|$, from condition (\ref{eq:infmaTK}) we now conclude that there exists a positive constant $D_{1}$ such that for all $q \in \mathbb{N}$ it holds that
\begin{equation}\label{eq:I1}
 \limsup_{n \to \infty} I_{1} \leq D_{1} \sum_{j=q}^{\infty}\mathrm{E} |C_{j}|.
\end{equation}

In order to estimate $I_{2}$ we consider separately the cases $\alpha \in (1,2)$ and $\alpha=1$. Assume first $\alpha \in (1,2)$. Applying Markov's inequality, the fact that the sequence $(C_{i})_{i \geq 0}$ is independent of $(Z_{i})$ and the stationarity of the sequence $(Z_{i})$ we obtain
\begin{eqnarray}\label{eq:alpha1}
\nonumber I_{2} & \leq &   \Pr \bigg(  \sum_{i=1}^{n} \bigg| \sum_{j=q}^{\infty} \frac{\widetilde{C}_{j}Z_{i-j}}{a_{n}} 1_{\{ |Z_{i-j}|>a_{n} \}} \bigg| > \frac{\epsilon}{2} \bigg)\\[0.4em]
 \nonumber & \leq & \frac{2}{\epsilon a_{n}}  \mathrm{E} \bigg(  \sum_{i=1}^{n} \bigg| \sum_{j=q}^{\infty} \widetilde{C}_{j}Z_{i-j} 1_{\{ |Z_{i-j}|>a_{n} \}} \bigg|  \bigg)\\[0.4em]
 & \leq & \frac{2n}{\epsilon a_{n}}  \sum_{j=q}^{\infty} \mathrm{E} |\widetilde{C}_{j}| \cdot \mathrm{E} \Big( |Z_{1}| 1 _{\{ |Z_{1}|>a_{n} \}} \Big)
\end{eqnarray}
By Karamata's theorem, as $n \to \infty$,
$$ \frac{n}{ a_{n}} \mathrm{E} \Big( |Z_{1}| 1 _{\{ |Z_{1}|>a_{n} \}} \Big) \to \frac{\alpha}{\alpha-1},$$
and hence from (\ref{eq:alpha1}) we conclude that there exists a positive constant $D_{2}$ such that
\begin{equation}\label{eq:I2a}
 \limsup_{n \to \infty} I_{2} \leq D_{2} \sum_{j=q}^{\infty} \mathrm{E} |C_{j}|.
\end{equation}
Now assume $\alpha =1$. Markov's inequality implies
$$ I_{2} \leq  \frac{2^{\delta}}{\epsilon^{\delta} a_{n}^{\delta}}  \mathrm{E} \bigg(  \sum_{i=1}^{n} \bigg| \sum_{j=q}^{\infty} \widetilde{C}_{j}Z_{i-j} 1_{\{ |Z_{i-j}|>a_{n} \}} \bigg| \bigg)^{\delta},$$
with $\delta$ as in relation (\ref{e:momcond}). Since $\delta < 1$, a double application of the triangle inequality $|\sum_{i=1}^{\infty}a_{i}|^{s} \leq \sum_{i=1}^{\infty}|a_{i}|^{s}$ with $s \in (0,1]$ yields
\begin{eqnarray*}
I_{2} & \leq & \frac{2^{\delta}}{\epsilon^{\delta} a_{n}^{\delta}} \sum_{i=1}^{n} \mathrm{E} \bigg( \bigg| \sum_{j=q}^{\infty} \widetilde{C}_{j}Z_{i-j} 1_{\{ |Z_{i-j}|>a_{n} \}} \bigg|^{\delta} \bigg)\\[0.4em]
      & \leq & \frac{2^{\delta}}{\epsilon^{\delta} a_{n}^{\delta}} \sum_{i=1}^{n} \sum_{j=q}^{\infty} \mathrm{E}  \bigg( \bigg| \widetilde{C}_{j}Z_{i-j} 1_{\{ |Z_{i-j}|>a_{n} \}} \bigg|^{\delta} \bigg).
\end{eqnarray*}
Using again the fact that $(C_{i})$ is independent of $(Z_{i})$ and the stationarity of $(Z_{i})$ we obtain
$$ I_{2} \leq \frac{2^{\delta} n}{\epsilon^{\delta} a_{n}^{\delta}} \mathrm{E} \Big( |Z_{1}|^{\delta} 1 _{\{ |Z_{1}|>a_{n} \}} \Big) \sum_{j=q}^{\infty} \mathrm{E} |\widetilde{C}_{j}|^{\delta}.$$
From this, since by Karamata's theorem
$$ \lim_{n \to \infty} \frac{n}{a_{n}^{\delta}} \mathrm{E} \Big( |Z_{1}|^{\delta} 1 _{\{ |Z_{1}|>a_{n} \}} \Big) = \frac{1}{1-\delta},$$
it follows that there exists a positive constant $D_{3}$ such that
\begin{equation*}\label{eq:I2b}
\limsup_{n \to \infty} I_{2} \leq D_{3} \sum_{j=q}^{\infty} \mathrm{E} |C_{j}|^{\delta}.
\end{equation*}
This together with (\ref{eq:I1I2}), (\ref{eq:I1}) and (\ref{eq:I2a}) shows that
$$ \limsup_{n \to \infty}\Pr[d_{M_{2}}(V_{n, q}, V_{n})> \epsilon] \leq D_{1} \sum_{j=q}^{\infty} \mathrm{E}|C_{j}| + (D_{2}+D_{3}) \sum_{j=q}^{\infty} \mathrm{E}|C_{j}|^{s},$$
 where
 $$ s = \left\{ \begin{array}{cc}
                                   \delta, & \quad \textrm{if} \ \alpha = 1,\\[0.4em]
                                   1, & \quad \textrm{if} \ \alpha \in (1,2).
                                 \end{array}\right.$$
Now, the dominated convergence theorem and conditions (\ref{e:momcond}) and (\ref{eq:infmaTK3}) yield (\ref{e:infSlutsky}). Therefore we again obtain $V_{n}(\,\cdot\,) \dto \widetilde{C} V(\,\cdot\,)$ in $(D[0,1], d_{M_{2}})$.

\end{proof}

\begin{rem}
If the sequence $(Z_{i})$ is an i.i.d.~or $\rho$--mixing sequence with $\sum_{i=1}^{\infty}\rho(2^{i}) < \infty$, where
$$ \rho(n) = \sup \{ | \textrm{corr}(f,g) | : f \in \mathrm{L}^{2}(\mathcal{F}_{1}^{k}), g \in \mathrm{L}^{2}(\mathcal{F}_{k+n}^{\infty}), k=1,2,\ldots \},$$
then it is known that condition (\ref{eq:infmaTK}) holds with $r=2$, see Tyran-Kami\'{n}ska~\cite{Ty10b}.

In the case when the sequence $(C_{j})$ is deterministic, conditions (\ref{e:mod1}) and (\ref{eq:infmaTK3}) can be dropped since they are implied by (\ref{e:momcond}). To see this note that by condition (\ref{e:momcond}) it holds that $|C_{j}|^{\delta} < 1$ for large $j$. Now since $|C_{j}|^{\delta x}$ is decreasing in $x$, it follows that for large $j$
$$ |C_{j}|^{\gamma} = (|C_{j}|^{\delta})^{\gamma/\delta} \leq |C_{j}|^{\delta},$$
and similarly $|C_{j}| \leq |C_{j}|^{\delta}$. This suffice to conclude that (\ref{e:mod1}) and (\ref{eq:infmaTK3}) hold. In general this does not hold when the coefficients are random (see for an example Krizmani\'{c}~\cite{Kr19}).
\end{rem}

\section*{Acknowledgment}
 This work has been supported in part by University of Rijeka research grants uniri-prirod-18-9 and uniri-pr-prirod-19-16 and by Croatian Science Foundation under the project IP-2019-04-1239.


\end{document}